\documentclass[11pt]{article}

\usepackage{amsfonts}
\usepackage{amssymb}
\usepackage{amsthm}
\usepackage{amsmath}
\usepackage{eucal}
\usepackage{bbm}
\usepackage[dvips]{epsfig}
\usepackage{tikz}
\usepackage{tikz-cd}
\usetikzlibrary{matrix,arrows,decorations.pathmorphing}
\tikzset{
commutative diagrams/.cd,
arrow style=tikz,
diagrams={>=angle 60}}

\newcommand{\overbar}[1]{\mkern 1.5mu\overline{\mkern-1.5mu#1\mkern-1.5mu}\mkern 1.5mu}

\newcommand{\nat}{\mathbb{N}}

\newcommand{\real}{\mathbb{R}}

\newcommand{\upitem}[1]{\item[\textup{({\em #1}\/)}]}
\newcommand{\ie}{{i.e.}}                                   

\newcommand{\viz}{{viz.}}
\newcommand{\cf}{{cf.\ }}
\newcommand{\etal}{{et al.}}
\newcommand{\prob}{\mathbb{P}}

\newcommand{\expec}{\mathbb{E}}

\newcommand{\ident}{I}
\newcommand{\indica}{\mathbbm{1}}

\newcommand{\calf}{\mathcal{F}}

\newcommand{\dom}{{\mathrm{dom}}}
\newcommand{\ran}{{\mathrm{ran}}}
\newcommand{\parfun}{\rightharpoonup}

\newcommand{\olf}{{\overbar{F}}}
\newcommand{\olg}{{\overbar{G}}}
\newcommand{\bfa}{{\mathbf A}}
\newcommand{\bfb}{{\mathbf B}}

\newcommand{\bfp}{{\mathbf P}}

\newcommand{\bfu}{{\mathbf u}}
\newcommand{\bfv}{{\mathbf v}}
\newcommand{\bfw}{{\mathbf w}}
\newcommand{\bfx}{{\mathbf x}}
\newcommand{\bfy}{{\mathbf y}}
\newcommand{\bfz}{{\mathbf z}}
\newcommand{\compose}{{\,\circ\,}}

\newcommand{\Section}[1]{\section{\hskip -1em.~~#1}}

\numberwithin{equation}{section}

\theoremstyle{plain}
\newtheorem{theorem}{Theorem}[section]
\newtheorem{corollary}[theorem]{Corollary}
\newtheorem{lemma}[theorem]{Lemma}
\newtheorem{proposition}[theorem]{Proposition}

\theoremstyle{definition}
\newtheorem*{definition}{Definition}

\theoremstyle{remark}
\newtheorem*{remark}{Remark}

\begin{document}

\title{Conditional Expectation Bounds with\\ Applications in Cryptography}

\author{
Kevin J. Compton\\
Computer Science and Engineering Division\\
University of Michigan - Ann Arbor\\
Ann Arbor, MI 48109-2212, USA\\
kjc@umich.edu}

\maketitle

\begin{abstract}
We  derive two conditional expectation bounds, which we use to simplify cryptographic security proofs. The first bound relates the expectation of a bounded random variable 
and  the  average of its conditional expectations  with
respect to a set of i.i.d.\ random objects.  It shows, under certain conditions, 
that  the conditional expectation average has a 
small  tail probability  when the  expectation of the random variable is sufficiently large. 
 It is used to simplify the proof  that the existence of weakly one-way functions implies the existence of strongly one-way functions.  The second bound relaxes the independence requirement on the random objects to give a result  that has applications to 
expander graph constructions in cryptography.  It is used to simplify the proof that there
is a security preserving reduction from weakly one-way functions to strongly one-way functions.  
To satisfy the hypothesis for this bound, we prove a hitting property for directed graphs that
are expander-permutation hybrids.
\end{abstract}

\Section{Introduction}

Let $Z$ be a  random variable on a probability space $(\Omega,\calf,\prob)$ with $0\leq Z\leq 1$ and $U_i$,
$0\leq i<t$, be {i.i.d.} random objects from the measurable space
$(\Omega,\calf)$ to the measurable space  $(\Omega',\calf')$.  Let $W$ be
the average of the conditional expectations
$W_i=\expec(Z|U_i)$.  From the law of iterated expectations and linearity of expectation, we
know that $\expec(W)=\expec(Z)$, but we would expect $W$ to be more concentrated around the mean.  Since it is not an average of {\em independent} random variables, however, this concentration is difficult to quantify.  We will derive a tail bound for $W$ which says, roughly, that  $W$ has a 
small  tail probability  when  $\expec(Z)$ is sufficiently large.  The contrapositive, 
that $\expec(Z)$ is small when $W$ has a large tail probability,
may be applied to simplify certain cryptographic security proofs, viz.,
the proof that the existence of weakly one-way functions implies the existence
of strongly one-way functions.  We will relax the independence requirement on the random objects
$U_i$ to obtain a somewhat weaker inequality that has applications to 
cryptographic expander graph constructions such as the one used to show that there
is a security preserving reduction from weakly one-way functions to strongly one-way functions.

We will assume the reader is familiar with terminology and notation from probability theory  found in standard texts, e.g.,  \cite{ash;doleans:00,Halmos:1974,loeve:1978}. Before proceeding,
we should point out that there are two definitions of  conditional expectation   in the literature. In most cases it does not matter which is used, but it does here.  Let $Y$ be a random variable on a probability space
$(\Omega,\calf,\prob)$ and $X$ be a measurable function from measurable space $(\Omega,\calf)$ to measurable space $(\Omega',\calf')$. We will take the definition that says 
$\expec(Y|X)$ is a random variable $V$ on 
$(\Omega',\calf',\prob_X)$ such that 
$$
\int_S V\,d\prob_X=\int_{\{X\in S\}}Y\,d\prob
$$
for all $S\in\calf'$.  (Here  $\prob_X$ is  the probability measure 
induced by $X$.)  This was the first rigorous definition of conditional expectation,
formulated by Kolmogorov in his seminal 1933 treatise
\cite{komolgorov:1933} and found in many current texts (e.g.,
\cite{ash;doleans:00,Halmos:1974}). Kolmogorov used the Radon-Nikodym Theorem to show that  such a $V$ 
exists whenever $\expec(Y)$ exists, and is unique in the sense that any two such $V$'s  differ only on a set of measure~0.  Some sources (e.g., \cite{feller:1971})  do not define  $\expec(Y|X)$ 
as  $V$,  but instead as $V\compose X$ (or an equivalent expression), which is a random variable on $(\Omega,\calf,\prob)$.
Lo\`{e}ve adopts the latter definition, but remarks that the original Komolgorov definition is the 
``usual interpretation''  (cf. \cite{loeve:1978}, p.\ 343).
$V\compose X$ is sometimes denoted $\expec(Y|\sigma(X))$, where $\sigma(X)$ is the $\sigma$-algebra induced by $X$.

The  situation described in the introductory paragraph arises naturally.  Suppose, for example, that we want
to approximate the probability of some event $S\subseteq\Omega$, where $\Omega$ is the solid
unit $n$-cube $[0,1]^n$.  If $n$ is large, it may not be feasible to generate random
points in $\Omega$  so instead we 
try to approximate the probability of a random lower-dimensional cross-section of $S$. We generate 
the cross-section
by randomly choosing a projection $U_i{:}\ \real^n\rightarrow\real^k$, where $k<n$,
from some fixed collection $U_0,U_1,\ldots,U_{t-1}$ of independent projections, then
randomly choosing some $x\in[0,1]^k$ and  intersecting $S$ 
with  the affine subspace $\{U_i=x\}$.  If we let $Z$ be the 
indicator function for $S$, $W_i=\expec(Z|U_i)$, and $W=(W_0+W_1+\cdots+W_{t-1})/t$,
then $W(x)$ is the probability (or expected volume) of this cross-section.

Take a specific example.  Let $n=t$ and 
fix $p$  between 0 and 1.   
Set  $S=[0,p^{1/t}]^t$ so $\expec(Z)=p$.
The $t$ projection functions $U_i$ given by
$U_i(y_0,y_1,\ldots,y_{t-1})=y_i$ are {i.i.d.} random objects and we
compute each $W_i=\expec(Z|U_i)$ to be
$$
  W_i(x)=\left\{
\begin{array}{ll}
 p^{1-1/t}, &  \mbox{if $x\leq p^{1/t}$},\\
 0, &  \mbox{otherwise}.
 \end{array}
\right.
$$
The random variables $W_i$   are identical, so $W$, their average, is the same function.
Now compare:
$Z$ takes the value 1 with probability $p$;  $W$ takes the value  $p^{1-1/t}$ with probability
$p^{1/t}$.  For large $t$, $W$ is more concentrated, taking a value closer to the
expectation $p$ on an event whose probability is $\expec(Z)^{1/t}$.  In fact,
for any $\varepsilon$ such that $0<\varepsilon < p^{1-1/t}$, $\prob_U\{W>\varepsilon \}=\expec(Z)^{1/t}$.

Our first  conditional expectation bound, Theorem~\ref{mainthm}(i),
says that something like  this holds in general. Since the random objects $U_i$ are
identically distributed we may write $U_i\sim U$ for some fixed $U$.   We show that when  $0\leq Z\leq 1$ and $0<\varepsilon <1$, 
\begin{equation}
\label{pone}
  \expec(Z)\leq\prob_U\{W> \varepsilon\}^t+t\varepsilon.
\end{equation}
This gives a lower bound for $\prob_U\{W> \varepsilon\}$ when  $t\varepsilon$ is
small, and, therefore, a tail bound 
$$
\prob_U\{W\leq  \varepsilon\}\leq 1- (\expec(Z)-t\varepsilon)^{1/t}.
$$

A further result, Theorem~\ref{secmainthm}(i),
dispenses with the
hypothesis that the random objects $U_i$ are identically distributed. The bound
is given in terms of product of tail probabilities $\prob_{U_i}\{W_i > \varepsilon_i\}$ plus
a correction term $t\varepsilon$,  where $\varepsilon$ is the average of the values $\varepsilon_i$.

 Our second  conditional expectation bound, Theorem~\ref{mainthm}(ii),
weakens the hypothesis  by replacing  independence of the random objects $U_i$ 
with  a property, introduced here, called $\beta$-{\em independence}
(where $0\leq\beta\leq 1$). In the case of identically distributed
random objects $U_i$, $\beta$-independence holds if for
all events $T_i\in \calf'$,setting $\prob\{U_i\in T_i\}=1-\nu_i$, we have
\begin{equation}
\label{betaind}
 \prob\{\bigwedge_{0\leq i < t}  U_i\in T_i\}\leq \prod_{0\leq i< t} (1-\beta\,\nu_i).
\end{equation} 
Under this hypothesis we conclude 
\begin{equation}
\label{ptwo}
  \expec(Z)\leq(\alpha+\beta\,\prob_U\{W> \varepsilon\})^t+t\varepsilon,
\end{equation}
where $\alpha=1-\beta$. It is not
difficult to see that 1-independence is equivalent to independence,
so that part (i) of Theorem~\ref{mainthm} is a special case of part (ii).
We should note that Theorem~\ref{mainthm}(ii) does not require full 
$\beta$-independence.  It suffices that (\ref{betaind}) holds when the events
$T_i$ are identical, i.e.,
\begin{equation}
\label{hitting}
  \prob\{\bigwedge_{0\leq i < t}  U_i\in T\}\leq (1-\beta\,\nu)^{t},
\end{equation}  
where  $\prob\{U_i\in T\}=1-\nu$.
However, we show a related result,  Theorem~\ref{secmainthm},
which dispenses with the condition that the random objects $U_i$ be identically
distributed, but requires full $\beta$-independence  given by (\ref{betaind}).

Inequalities (\ref{pone}) and (\ref{ptwo}) are not difficult to prove.
We regard them as useful probabilistic tools   similar  to Chernoff bounds. But whereas we
use Chernoff bounds to show that an efficient  probabilistic algorithm has a
high probability of returning the correct result, 
we use conditional expectation bounds to show  that
all efficient probabilistic  algorithms have a low probability of breaking a cryptographic 
construction.   One of the principal contributions of this paper is to show how conditional
expectation bounds  can simplify certain security proofs.

The earliest example of such a security proof
deals with  Yao's construction \cite{yao:1982}  of  a strongly  
one-way function as a   direct power of a weakly one-way
function. (See Section~\ref{oneway} for precise definitions.)  The idea is simple.  From $F$, a  weakly one-way function,
define
$$
  F'(x_0x_1\cdots x_{t-1})=(F(x_0),F(x_1),\ldots, F(x_{t-1})),
$$
where $x_0,x_1,\ldots ,x_{t-1}\in\{0,1\}^n$ and  $t=t(n)$ is suitably chosen polynomial.  
The proof  that $F'$ is  strongly one-way  is not so simple. 

The first proof of this result appeared  in  an online draft of
a text by Goldreich \cite{goldreich:2001}.
Let us describe it in terms of a
reduction between cryptographic primitives
(as found in \cite{luby:1996,crescenzo:99,lin:2005}).   
For one-way functions, a
 {\em reduction} is a pair $({\mathcal R},{\mathcal R^*})$
where $\mathcal R$  is an efficient transformation taking 
$F$ to $F'$  and $\mathcal R^*$ is an efficient transformation
taking each  randomized function  $G'$ attempting 
to invert $F'$ to  a randomized function
$G$ attempting to invert~$F$.    
There is also a condition, detailed in section~\ref{oneway},  relating  the  probability that $G$ inverts $F$  to  the probability that $G'$ inverts $F'$.  In this situation, we say  $({\mathcal R},{\mathcal R^*})$ is
a reduction from  $F$ to $F'$ (although it would
be more accurate  to say
it is  a reduction of the invertibility problem for $F$ to
the invertibility problem for $F'$).

Goldreich's proof implicitly gives such a reduction.
${\mathcal R}$ is the  the direct power construction.  The crux of the proof is to specify
${\mathcal R^*}$ so that if  $G'$ is probabilistic polynomial time (p.p.t.) computable then so is   $G$, and  whenever   $G$ inverts $F$ with probability  significantly less 
than~1,    $G'$  inverts $F'$ with  negligible probability.  If
$F$ is  weakly one-way, every {p.p.t.} $G$ inverts $F$ with probability significantly less than~1,
and, hence, every {p.p.t.} $G'$ inverts $F'$ with negligible probability.  Therefore,
 $F'$ is  strongly one-way.
The proof of this in \cite{goldreich:2001} relies on a number of technical subsidiary results running to several pages.  Using the conditional expectation inequality (\ref{pone}), we are able to significantly simplify the proof (cf. Theorem~\ref{weaktostrong}).

Goldreich \etal\ \cite{goldreich:1990} later pointed
out a drawback to the direct power construction: the associated reduction
is not   security preserving  (see section~\ref{security} for the precise
definition).
The reason is that the forward transformation
${\mathcal R}$ takes $F$ with input size $n$ to $F'$ with
input size  $n\,t(n)$;
${\mathcal R^*}$, then, takes $G'$ with input size $n\,t(n)$ to $G$ with input size $n$.  It follows that if $S(n)$ is the security 
of $F$ against $G$ and $S'(n)$ is the security of $F'$ against
$G'$, then $S(n)$ is of the same order as  $S'(n\,t(n))$, which is much
larger than $S'(n)$ when $S'$ has superpolynomial growth.
They remedied this deficiency by replacing  the direct
power  with  a more elaborate  construction employing expander graphs
to control input size blowup.   
They give a security preserving reduction in the restricted case where $F$  is a weakly one-way permutation and
 $F'$ is a strongly one-way 
permutation, but now rather than $S'(n\,t(n))$, the security of $F'$ against $G'$ is   $S'(n+\omega(\log n))$, which is of the same order as $S'(n)$ (in a sense made precise in section~\ref{security}).
In this case, (\ref{pone}) does not suffice, so we use (\ref{ptwo}).
Again, the original proof relies on many subsidiary technical results which we bypass by means of
the conditional expectation bound (\ref{ptwo}).
To apply it, though, we must satisfy  the $\beta$-independence condition.

The hitting property for expander graphs is closely related to $\beta$-independence.  The version
of hitting property in \cite{vadhan:2012} is equivalent to the statement that if $U_i$ is the $i$-th
vertex along a random walk in an expander graph ${\mathcal G}=(V,E)$ whose spectral gap is $\beta$, $T\subseteq V$, and $|T|/|V|=1-\nu$, then (\ref{hitting}) holds.  As we noted, (\ref{hitting})  is weaker than the
full $\beta$-independence definition  (\ref{betaind}) where we have $T_i\subseteq V$ for each
$i<t$, but it is possible to prove this stronger version of the hitting property for expander
graphs; this result does not seem to be stated anywhere in the literature (but does follow
from Theorem~3.11 of \cite{hoory;linial;wigderson;2006} using the Cauchy-Schwarz inequality).  In section~\ref{betasec}
we will show an even stronger result.  We will show that (\ref{ptwo}) holds for
random walks on {\em directed graphs} that are hybrids formed from expander graphs and vertex
permutations.  This significantly simplifies the proof for the  reduction from weakly to strongly one-way
permutations in \cite{goldreich:1990,goldreich:2001}.

The outline of the paper is as follows.
 In Section~\ref{bounds} we prove the
conditional expectation bounds.   In Section~\ref{betasec} we show
that the  $\beta$-independence property
(\ref{betaind}) holds in  hybrid expander-permutation directed graphs.  
 In Section~\ref{oneway} we define and discuss the
properties of  weakly and strongly 
one-way functions.  In Section~\ref{application} we 
use the  conditional expectation 
inequality (\ref{pone}) to simplify Goldreich's proof that
existence of weakly one-way functions implies the 
existence of strongly one-way functions.
In Section~\ref{security} we use  the  conditional expectation 
inequality (\ref{ptwo}) to show
that the hybrid expander-permutation directed graph construction
yields  a strongly one-way permutation when the expander graph is
from a  certain  fully explicit family expander graphs.  
In section~\ref{security}  we show that there is a security
preserving reduction from weakly to strongly one-way permutations.
Our result is stronger than the one in
 Goldreich \etal\ \cite{goldreich:1990} because they require
 that their expander graph families have fully explicit polynomial-time computable 
edge-colorings.
 We will show that 
 the edge-coloring assumption is not necessary.  In section~\ref{conclusion}, we conclude with some
 brief remarks.

\Section{Conditional Expectation Bounds}
\label{bounds}

Before stating our conditional expectation bounds, we recall some basic terminology and notation
from probability theory.  A {\em random object} on $(\Omega,\calf,\prob)$ is a measurable function
$X$ from measurable algebra $(\Omega,\calf)$ to  measurable algebra
$(\Omega',\calf')$. 
When $\Omega' =\real$ and $\calf'$ is the $\sigma$-algebra of Borel sets,
$X$ is  a {\em random variable}.  The {\em indicator
function} for an event $S$ is
$$
\indica_S(x)=\left\{
\begin{array}{ll}
 1, &  \mbox{if $x\in S$},\\
 0, &  \mbox{otherwise},
 \end{array}
\right.
$$
 $\expec(Z)$ is the expectation of a random variable $Z$
and $\expec_S(Z)$ denotes $\expec(\indica_S\cdot Z)$.
A random object $U{:}\ (\Omega,\calf)\rightarrow(\Omega',\calf')$ {\em induces} the  probability measure $\prob_U$ (also called the  {\em distribution} of $U$) on $(\Omega',\calf')$ given by
$
  \prob_U(T)=\prob\{U\in T\}.
$

As we noted earlier, we introduce here a generalization of the notion of independence for random objects.
Let $0\leq\beta\leq 1$ and $\alpha=1-\beta$.  
The random objects $U_i{:}\ (\Omega,\calf)\rightarrow(\Omega'_i,\calf_i')$, $i<t$,  are {\em $\beta$-independent}
if for all events $T_i\in\calf_i'$ with $\prob_{U_i}(T_i)=\mu_i=1-\nu_i$,
$$\prob\{\bigwedge_{0\leq i < t}  U_i\in T_i\}\leq \prod_{0\leq i< t} (1-\beta\,\nu_i)$$
or, equivalently,
$$\prob\{\bigwedge_{0\leq i < t}  U_i\in T_i\}\leq \prod_{0\leq i< t} (\alpha+\beta\,\mu_i).$$
If, in addition, the random objects $U_i$ are identically distributed, they are said to be $\beta$-{i.i.d.}

We now state the conditional expectation bounds.

\begin{theorem}
\label{mainthm}
Let $Z$ be a random variable such that $0\leq Z\leq 1$, and $$U_i{:}\ (\Omega,\calf)\rightarrow(\Omega' ,\calf' ),  i<t,$$
be identically distributed
with $U_i\sim U$ for some fixed $U$.  Set $W_i=\expec(Z|U_i)$, $W=(W_0+W_1+\cdots+W_{t-1})/t$ and take any $0<\varepsilon <1$.
\begin{itemize}
\upitem{i} If the random objects  $U_i$ are independent,
$$
  \expec(Z)\leq\prob_{U}\{W> \varepsilon\}^t+t\varepsilon.
$$

\upitem{ii}  If the random objects $U_i$ are $\beta$-independent,
$$
  \expec(Z)\leq(\alpha+\beta\,\prob_{U}\{W> \varepsilon\})^t+t\varepsilon,
$$
where $\alpha=1-\beta$.
\end{itemize}
\end{theorem}

\begin{proof}
Let $T=\{W> \varepsilon\}$, $S_i=\{U_i\in T\}$, and $S=\bigcap_{ i<t} S_i$.
By the  the law of total expectation, 
\begin{equation}
\label{mainstep}
  \expec(Z) = \expec(Z|S)\,\prob(S)+\expec(Z|\overbar{S})\,\prob(\overbar{S}),
\end{equation}
Bound the right side of this equation as follows.
$\expec(Z|S)\leq 1$ since $Z\leq 1$.
Independence of the random objects $U_i$ in part (i) of the theorem implies
$$\prob(S)=\prob(\bigcap_{ i<t}S_i)=\prod_{ i<t}\prob_{U_i}(T)
=\prob_{U}\{W> \varepsilon\}^t,$$
and $\beta$-independence in part (ii) implies
$$
\prob(S)\leq  \prod_{ i<t}(\alpha + \beta\, \prob_{U_i}(T))
=(\alpha+\beta\,\prob_{U}\{W> \varepsilon\})^t.
$$
From  $\overbar{S}=\bigcup_{ i<t} \overbar{S}_i$  we have 
\begin{eqnarray*}
  \expec(Z|\overbar{S})\,\prob(\overbar{S})    &=&\expec_{\overbar{S}}(Z)\\
                                             &\leq& \sum_{ i<t}\expec_{\overbar{S}_i}(Z),
 \end{eqnarray*}
 which,  by the law of iterated expectations, is equal to
 \begin{eqnarray*}
\sum_{ i<t}\expec_{\overbar{T}}(\expec(Z|U_i))
                                             &=& \sum_{ i<t}\expec_{\{W\leq \varepsilon \}}(W_i)\\
                                             &=& \expec_{\{W\leq \varepsilon \}}(\sum_{ i<t}W_i)\\
                                             &=& \expec_{\{W\leq \varepsilon \}}(t W)\\
                                             &\leq& t\varepsilon,
\end{eqnarray*}
since $\overbar{S}_i=\{U_i\in \overbar{T}\}$ and $\overbar{T}=\{W\leq \varepsilon \}$.

Substitution into (\ref{mainstep}) completes the proof.
\end{proof}

\begin{remark}
As we noted in the introduction, we may
relax the hypotheses of  Theorem~\ref{mainthm} so that rather than 
independence in part (i), we  require only
that for every event $T$,
$$
  \prob\{\bigwedge_{ i<t}U_i\in T\}=\prod_{ i<t}\prob_{U_i} (T),
$$
and similarly for $\beta$-independence in part (ii).
\end{remark}

We do not use the  following theorem in this paper,
but it may prove useful in other contexts where the random variables
$U_i$ are not identically distributed.

\begin{theorem}
\label{secmainthm}
Let $Z$ be a random variable such that $0\leq Z\leq 1$, $$U_i{:}\ (\Omega,\calf)\rightarrow(\Omega_i' ,\calf_i' ),  i<t,$$
be random objects with $W_i=\expec(Z|U_i)$, $0<\varepsilon _i<1$ for $ i<t$,
and $\varepsilon=(\sum_{ i< t}\varepsilon_i)/t$.
\begin{itemize}
\upitem{i} If the random objects $U_i$ are independent,
$$
  \expec(Z)\leq\prod_{ i<t}\prob_{U_i}\{W_i > \varepsilon_i\}+t\varepsilon.
$$

\upitem{ii}  If the random objects $U_i$ are $\beta$-independent,
$$
  \expec(Z)\leq\prod_{ i<t}(\alpha+\beta\,\prob_{U_i}\{W_i> \varepsilon_i\})+t\varepsilon,
$$
where $\alpha=1-\beta$.
\end{itemize}
\end{theorem}

\begin{proof}
Let $T_i=\{W_i> \varepsilon_i\}$, $S_i=\{U_i\in T_i\}$, and $S=\bigcap_{ i<t} S_i$.
As in the previous theorem, we bound the right side of
\begin{equation*}
  \expec(Z) = \expec(Z|S)\,\prob(S)+\expec(Z|\overbar{S})\,\prob(\overbar{S}).
\end{equation*}
As before,  $\expec(Z|S)\leq 1$.  Also, for part (i) of the theorem,
$$\prob(S)=\prob(\bigcap_{ i<t}S_i)=\prod_{ i<t}\prob_{U_i}(T_i),$$
and for  part (ii),
$$
\prob(S)\leq  \prod_{ i<t}(\alpha + \beta\, \prob_{U_i}(T_i)).
$$
We have
$$
  \expec(Z|\overbar{S})\prob(\overbar{S})  \leq \sum_{ i<t}\expec_{\overbar{S}_i}(Z),
$$
This last summation is equal to
 \begin{eqnarray*}
\sum_{ i<t}\expec_{\overbar{T}_i}(\expec(Z|U_i))
                                             &=& \sum_{ i<t}\expec_{\{W_i\leq \varepsilon_i\}}(W_i)\\
                                             &\leq& \sum_{ i<t}\varepsilon_i,
\end{eqnarray*}
since $\overbar{S}_i=\{U_i\in \overbar{T}_i\}$ and $\overbar{T}_i=\{W_i\leq \varepsilon_i\}$.
The theorem  follows by substitution.
\end{proof}

\Section{Expander-Permutation Graphs and $\beta$-Independence.}
\label{betasec}

Theorem~\ref{beta}, the main result of this section, provides a natural construction of
$\beta$-independent random objects based on  hybrid expander-permutation directed graphs.  

Let  $\bfa$ be a symmetric, real-valued  matrix of dimension $N$.  In particular,
$\bfa$ is Hermitian
(\cf Horn and Johnson \cite{horn;johnson:1985} for basic results
concerning Hermitian matrices), so it
has $N$ real eigenvalues.  List them (with
repetitions according to multiplicities)  in nonincreasing order:
$$\lambda_0(\bfa)\geq\lambda_1(\bfa)\cdots\geq\lambda_{N-1}(\bfa).$$
  We write $\lambda_i$ rather than $\lambda_i(\bfa)$ when matrix
$\bfa$ is clear from context.  Because $\bfa$ is Hermitian, there is an orthonormal basis
$\bfu_0,\bfu_1,\ldots,\bfu_{N-1}\in\real^n$, where $\bfu_i$ is a left (row) eigenvector
associated with eigenvalue $\lambda_i(\bfa)$.   Since $\bfa$ is symmetric, the respective right eigenvectors are transposes of the left eigenvectors.

We will consider only real-valued vectors.  We write
vectors in lower case boldface  and  denote the transpose of
vector $\bfv$ by $\bfv^T$, so the inner product  of row vectors $\bfv$ and
$\bfw$ is $\langle\bfv,\bfw\rangle=\bfv\cdot\bfw^T$ and
$|\bfv|^2=\langle\bfv,\bfv\rangle$.   The Cauchy-Schwarz
inequality states that $|\langle\bfv,\bfw\rangle|\leq|\bfv|\,|\bfw|$.

Let  $\mathcal{G}=(V,E)$ be a $d$-regular undirected graph and $\bfa$ be
its {\em transition matrix}; \ie, $\bfa=(a_{i,j})$, where
$$
a_{i,j}=\left\{
\begin{array}{ll}
 1/d, &  \mbox{if there is an edge from vertex $i$ to vertex $j$},\\
 0, &  \mbox{otherwise}.
 \end{array}
\right.
$$
$\bfa$ is symmetric,   nonnegative, and doubly stochastic.  If $\mathcal{G}$
is connected, then $\bfa$ is irreducible (\cf Seneta \cite{seneta:1981}
for basic results concerning  nonnegative matrices).
By the Perron-Frobenius Theorem, the
largest eigenvalue of $\bfa$ is $\lambda_0=1$, the common row sum.
Also, this  is a simple eigenvalue so $\lambda_0>\lambda_1$ and
no other eigenvalue is larger in magnitude than $\lambda_0$;
it follows that $\lambda_{N-1}\geq -1$.  If  $\mathcal{G}$ is not bipartite, then $\lambda_{N-1}>-1$.
Under these conditions, define $\alpha=\alpha(\bfa)$ to be the second largest eigenvalue magnitude of $\bfa$, and
 $\beta$ to be the {\em spectral gap}
of $\bfa$, \ie, the difference between the largest and second
largest eigenvalue magnitudes.  That is,
$
  \alpha=\max(|\lambda_1(\bfa)|,|\lambda_{N-1}(\bfa)|)$ and 
 $ \beta = 1-\alpha$.

\begin{definition}
A connected, non-bipartite graph $\mathcal{G}=(V,E)$ is an {\em $(N,d,\alpha)$-expander graph} if $|V|=N$,
$\mathcal{G}$ is $d$-regular, and the second largest eigenvalue magnitude of its transition
matrix is at most $\alpha$.
\end{definition}

For the remainder of the section,
$\bfa$ is the transition matrix for an $(N,d,\alpha)$-expander
graph $\mathcal{G}$, $\alpha+\beta=1$, and vectors $\bfu_0,\bfu_1,\ldots,\bfu_{N-1}\in\real^n$ form
an orthonormal basis, where $\bfu_i$ is a left eigenvector
associated with eigenvalue $\lambda_i(\bfa)$.  We take
$\bfu_0$ to be $N^{-1/2}(1,1,\ldots,1)$.
For reference, we collect a few simple facts.  Here the {\em projection matrix}  for a set $S\subseteq \{0,1,\ldots,N-1\}$ is the matrix
$\bfp=(p_{i,j})$ where $p_{i,i}=1$ if $i\in S$ and all other entries $p_{i,j}$
are~0.

\begin{proposition}
\label{ortho}
Let $V_0$ be the subspace of $\real^N$ spanned by $\bfu_0$ and
$V_1$ its orthogonal space, the subspace
spanned by $\bfu_1,\ldots,\bfu_{N-1}$.
\begin{itemize}
\renewcommand{\theenumi}{(\roman{enumi})}
\upitem{i}  $V_0$ and $V_1$ are invariant under the action of $\bfa$.

\upitem{ii} Every vector $\bfw\in\real^N$ can be decomposed as a
sum of two vectors $\bfw=\bfx+\bfy$ where $\bfx\in V_0$ and $\bfy\in V_1$.

\upitem{iii} If $\bfx\in V_0$, then $\bfx\bfa=\bfx$.  If $\bfy\in V_1$, then
$|\bfy\bfa|\leq\alpha|\bfy|$.

\upitem{iv} If $\bfp$ is the projection matrix for $T\subseteq \{0,1,\ldots,N-1\}$, then $|\bfu_0\bfp|^2=|T|/N$.
\end{itemize}
\end{proposition}

Random walks on expander graphs with spectral gap $\beta$ give rise to $\beta$-{i.i.d.} random objects.
A {\em $t$-walk} on $\mathcal{G}$ is a $(t+1)$-tuple $\bfy=(y_0,y_1,\ldots,y_t)$
of (not necessarily distinct) vertices such that  $\{y_i,y_{i+1}\}\in E$ for $ i<t$.
Let $\Omega$ be the set of all $t$-walks on $\mathcal{G}$ and $\prob$ be the uniform
probability measure on $\Omega$.  We shall see that the projection functions $U_i$,
defined by $U_i(\bfy)=y_i$  are $\beta$-independent. 

Of course, we
could just take the projection functions 
on the Cartesian product $V^{t+1}$, but many applications require reducing 
 the number of bits needed to represent a sample point.  We  generate a $t$-walk in a unique way by
picking an initial vertex $y_0$  and then choosing each successive vertex
$y_{i+1}$ by traversing one of the $d$ edges incident with $y_i$.  Representing a $t$-walk
this way requires  $\lceil\lg N\rceil + t\lceil \lg d\rceil $ bits while
representing a point in $V^{t+1}$ requires $(t+1)\lceil\lg N\rceil$ bits. 

We now extend this idea to random walks on directed graphs that are expander-permutation hybrids.
Take an $(N,d,\alpha)$-expander graph ${\mathcal G}=(V,E)$ together with  a permutation $F$ on $V$
and define  the composition 
$$
  E'=F\compose E=\{(u,F(v))\mid\{u,v\}\in E\}
$$
obtain  a  directed graph $\mathcal{G'}=(V,E')$ (possibly with loops).
(Goldreich 
\etal\  implicitly use $E'=E\compose F$, but either construction works.)
$\mathcal{G'}$ is a {\em $d$-regular} directed graph in the sense that
every vertex has {\em both} indegree and outdegree $d$. 
Let $\Omega$ be the set of {\em directed $t$-walks}   $\bfy=(y_0,y_1,\ldots,y_t)$,
where $(y_i,y_{i+1})\in E'$ for $ i<t$, 
$\prob$  be the uniform
probability measure on $\Omega$, and $U_i(\bfy)=y_i$, as before.  Theorem~\ref{beta} below
shows that the projection functions $U_i$ are $\beta$-independent.\footnote{Some
sources assert that since $F$ is a permutation, $\mathcal{G}$ and $\mathcal{G'}$
have the same mixing properties.  This is true, but since the adjacency matrix of
$\mathcal{G'}$ is not Hermitian, this assertion does not follow directly.}

$\mathcal G'$  has transition matrix $\bfa'=\bfa\bfb$,
where $\bfb$ is the permutation matrix for $F$ (\ie, the $(i,j)$ entry of $\bfb$ is~1 if $F(i)=j$ and~0 otherwise).  $\bfa'$  represents a step of a random walk on $\mathcal G'$.  When it acts
on a row vector representing a probability distribution on $V$, the result is the succeeding  probability distribution along the
random walk.
$\bfa'$ may also act on  an {\em improper} probability distribution, i.e.,
one whose coordinates are nonnegative and sum to at most~1.   Let
$\Omega$ be the set of all $t$-walks on $\mathcal G'$ and $\prob$ be the 
uniform probability measure on $\Omega$.  
  The {\em terminal probability vector}  of $B\subseteq \Omega$ is
a  vector $(p_0,p_1,\ldots,p_{N-1})$ where  $p_j=\prob(B\cap \{U_t=j\})$
for each vertex $j\in V$.
This is an improper  probability distribution whose coordinates sum to $\prob(S)$.

The following result is well known and follows by standard techniques \cite{hoory;linial;wigderson;2006}.

\begin{proposition}
\label{fundamental}
For $\mathcal G'$ as above, let  $T_i\subseteq V$ and $\bfp_i$ be the 
projection matrix for $T_i$ for $i\leq t$.  Then the terminal probability vector
for  $\{\bigwedge_{ i\leq t} U_i\in T_i\}$ is given by
$$
\bfv\,\bfp_0\bfa'\,\bfp_1\cdots\bfp_{t-1}\bfa'\,\bfp_t,
$$
where $\bfv=N^{-1}(1,1,\ldots,1)$.
\end{proposition}

The idea here is clear: the vector $\bfv$ is the initial probability vector, $\bfa'$ represents a step
along the random walk, and each $\bfp_i$ eliminates paths whose $i$-th vertex is not in $T_i$.

Now we require two lemmas.

\begin{lemma}
\label{eigenbound}
Let $\bfa$, $N$, $d$, $\alpha$, and $\beta$ be as above. Take projection matrices $\bfp$ and $\bfp'$  for
$T$ and $T'\subseteq\{0,1,\ldots,N-1\}$ and set $\mu=|T|/N$, $\mu'=|T'|/N$. Then for all for all  $\bfv\in\real^N$,
$$
|\bfv\,\bfp\bfa\bfp'|\leq (\alpha+\beta\mu)^{1/2}(\alpha+\beta\mu')^{1/2}|\bfv|
$$
\end{lemma}

\begin{proof}
Let $\bfz$ be a unit vector parallel to $\bfv\bfp\bfa\bfp'$ and put $\bfw=\bfv\bfp$, $\bfw'=\bfz\bfp'$
so
\begin{eqnarray*}
 |\bfv\bfp\bfa\bfp'|&=&\langle \bfv\bfp\bfa\bfp',\bfz \rangle\\
                           &=&\langle \bfv\bfp\bfa,\bfz\bfp' \rangle\\
                           &=&\langle \bfw\bfa,\bfw'\rangle\\
                           &=& \langle \bfx\bfa,\bfx'\rangle + \langle \bfx\bfa,\bfy'\rangle
+ \langle \bfy\bfa,\bfx'\rangle + \langle \bfy\bfa,\bfy'\rangle,
 \end{eqnarray*}
where $\bfw=\bfx+\bfy$ and $\bfw'=\bfx'+\bfy'$ are the decompositions  given by Proposition~\ref{ortho}(ii).
Proposition~\ref{ortho}(i) and (iii) imply
$\langle \bfx\bfa,\bfy'\rangle=\langle \bfy\bfa,\bfx'\rangle=0$,
$|\langle \bfx\bfa,\bfx'\rangle|=|\langle \bfx,\bfx'\rangle|= |\bfx|\,|\bfx'|$,
and $|\langle \bfy\bfa,\bfy'\rangle|\leq |\bfy\bfa|\,|\bfy'|\leq \alpha\,|\bfy|\,|\bfy'|$, 
so
\begin{equation}
\label{ineq1}
|\bfv\bfp\bfa\bfp'|\leq |\bfx|\,|\bfx'|+\alpha\,|\bfy|\,|\bfy'|.
\end{equation}
By the Cauchy-Schwarz inequality
\begin{eqnarray*}
|\bfx|\,|\bfx'|+|\bfy|\,|\bfy'|&\leq&(|\bfx|^2+|\bfy|^2)^{1/2}(|\bfx'|^2+|\bfy'|^2)^{1/2}\\
&=&|\bfw|\,|\bfw'|
\end{eqnarray*}
so $|\bfy|\,|\bfy'|\leq|\bfw|\,|\bfw'| - |\bfx|\,|\bfx'|$.  Thus, substituting into (\ref{ineq1}),
\begin{equation}
\label{ineq2}
|\bfv\bfp\bfa\bfp'|
\leq\alpha\,|\bfw|\,|\bfw'|+\beta\,|\bfx|\,|\bfx'|.
\end{equation}

We must bound  $|\bfw|,|\bfw'|,|\bfx|$, and $|\bfx'|$.
First, $|\bfw|=|\bfv\bfp|\leq|\bfv|$ and $|\bfw'|=|\bfz\bfp'|\leq|\bfz|=1$.  
Next, we know $\bfx+\bfy=\bfv\bfp$.  Take the inner product of both sides with $\bfu_0$  and
observe that  $\bfy$ and 
$\bfu_0$ are orthogonal
so
$\langle \bfx,\bfu_0\rangle=\langle \bfv\bfp,\bfu_0\rangle=\langle \bfv,\bfu_0\bfp\rangle$.
Then $|\bfx|=|\langle \bfv,\bfu_0\bfp\rangle|\leq |\bfv|\,|\bfu_0\bfp|=\mu^{1/2}|\bfv|$ by
Proposition~\ref{ortho}(iv).
Also, we know $\bfx'+\bfy'=\bfz\bfp'$.  Take the inner product of both sides with $\bfu_0$ as before.
This time,  $|\bfx'|\leq |\bfz|\,|\bfu_0\bfp'|=(\mu')^{1/2}$, again by Proposition~\ref{ortho}(iv).

 Substituting these four bounds into (\ref{ineq2}) gives
$$
  |\bfv\bfp\bfa\bfp'|\leq (\alpha+\beta\left(\mu\mu'\right)^{1/2})|\bfv|.
$$
By the Cauchy-Schwarz inequality,
\begin{eqnarray*}
\alpha+\beta\left(\mu\mu'\right)^{1/2}&=&{\alpha}^{1/2}{\alpha}^{1/2}+(\beta\mu)^{1/2}(\beta\mu')^{1/2}\\
                                   &\leq&(\alpha+\beta\mu)^{1/2}(\alpha+\beta\mu')^{1/2}
\end{eqnarray*}
which completes the proof.
\end{proof}

\begin{lemma}
\label{newbound}
Let $\bfa'$, $N$, $d$, $\alpha$, and $\beta$ be as above.  Take projection matrices $\bfp$ and $\bfp'$  for
$T$ and $T'\subseteq\{0,1,\ldots,N-1\}$ and set $\mu=|T|/N$, $\mu'=|T'|/N$.  Then for all for all  $\bfv\in\real^N$,
$$
|\bfv\,\bfp\bfa'\bfp'|\leq (\alpha+\beta\mu)^{1/2}(\alpha+\beta\mu')^{1/2}|\bfv|.
$$
\end{lemma}

\begin{proof}
We have
\begin{eqnarray*}
\bfp\bfa'\bfp' &=& \bfp\bfa\bfb\bfp'\\
               &=& \bfp\bfa(\bfb\bfp'\bfb^{-1})\bfb\\
               &=&  \bfp\bfa\bfp''\bfb,
\end{eqnarray*}
where $\bfp''=\bfb\bfp'\bfb^{-1}$.  It is easy to see that $\bfp''$ is the projection matrix for $F^{-1}[T']$.
Now $|F^{-1}[T']|/N=|T'|/N=\mu'$, so
by Lemma~\ref{eigenbound}, $|\bfv\,\bfp\bfa\bfp''|\leq (\alpha+\beta\mu)^{1/2}(\alpha+\beta\mu')^{1/2}|\bfv|$.
Thus, since $\bfb$ is a permutation matrix,
\begin{eqnarray*}
|\bfv\,\bfp\bfa'\bfp| &=& |\bfv\,\bfp\bfa\bfp''\bfb|\\
                                  &=& |\bfv\,\bfp\bfa\bfp''|\\
                                  &\leq& (\alpha+\beta\mu)^{1/2}(\alpha+\beta\mu')^{1/2}|\bfv|
\end{eqnarray*}
by the previous lemma.
\end{proof}

We come now to the main result of this section giving the construction of $\beta$-independent
random objects.

\begin{theorem}
\label{beta}
Let $\mathcal{G}'$  be a hybrid expander-permutation directed graph where the expander graph
has spectral radius $\beta$, $\Omega$ be  the set of all directed
$t$-walks  in $\mathcal{G'}$, and $\prob$ be the uniform probability measure on $\Omega$.
Then the projection functions  $U_i$ are $\beta$-independent random objects.
\end{theorem}

\begin{proof}
We need to show that
for all $T_i\subseteq V$, $ i\leq t$,
\begin{equation}
\label{bind}
  \prob\{\bigwedge_{ i\leq t} U_i\in T_i\}\leq \prod_{ i\leq t} (\alpha+\beta\mu_i),
\end{equation}
where $\mu_i=\prob\{U_i\in T_i\}$.

We claim that it is enough to show this in the special case where $T_0=T_t=V$.  In this case $\{U_0\in T_0\}=\{U_t\in T_t\}=\Omega$, $\mu_0=\mu_t=1$,
and $\alpha+\beta\mu_0=\alpha+\beta\mu_t=1$.  In effect, this eliminates constraints on the initial and terminal vertices of walks.
Then (assuming $t>2$) we may delete the initial and terminal vertices to obtain (\ref{bind}) for $(t-2)$-walks rather than $t$-walks.
That is, letting $\prob'$ be the uniform
measure on the set of $(t-2)$-walks,
$$
 \prob'\{\bigwedge_{1\leq i\leq t-1} U_i\in T_i\}=\prob\{\bigwedge_{ 0\leq i\leq t} U_i\in T_i\}=\prod_{1\leq i\leq t-1} (\alpha+\beta\mu_i).
$$

From Proposition~\ref{fundamental} we know that the terminal probability vector for $\{\bigwedge_{ 0\leq i\leq t} U_i\in T_i\}$ is given by
\begin{equation}
\label{probvec}
\bfv\,\bfp_0\bfa'\,\bfp_1\cdots\bfp_{t-2}\bfa'\,\bfp_{t-1}.
\end{equation}
where $\bfv=N^{-1}(1,1,\ldots,1)$.
We may rewrite this as
$$
\bfv\,(\bfp_0\bfa'\,\bfp_1)(\bfp_1\bfa'\,\bfp_2)\cdots(\bfp_{t-1}\bfa'\,\bfp_t)
$$
since $\bfp_i\bfp_i=\bfp_i$ for $1\leq i\leq t-1$.
By Lemma~\ref{newbound}, multiplication by $\bfp_i\bfa'\bfp_{i+1}$
changes the magnitude of a vector by at most a factor of $(\alpha+\beta\mu_i)^{1/2}(\alpha+\beta\mu_{i+1})^{1/2}$,
so (\ref{probvec}) is bounded in magnitude by
$$
   |\bfv|\prod_{ i\leq t-1} (\alpha+\beta\mu_i)^{1/2}(\alpha+\beta\mu_{i+1})^{1/2}
$$
which is equal to
$$
 |\bfv|\,(\alpha+\beta\mu_0)^{1/2}(\alpha+\beta\mu_t)^{1/2}\prod_{1\leq i\leq t-1} (\alpha+\beta\mu_i).
$$
By assumption, $\alpha+\beta\mu_0=\alpha+\beta\mu_t=1$, so we can further simplify this to
$$
 |\bfv|\,\prod_{ i\leq t} (\alpha+\beta\mu_i).
$$

Finally, compute the probability  of $\{\bigwedge_{ i\leq t} U_i\in T_i\}$ by summing the coordinates of its
terminal probability vector (\ref{probvec});  we do this by taking the inner product 
with  the vector $\bfu=(1,1,\ldots,1)^T$. The result is bounded in magnitude by
$$
 |\bfu|\,|\bfv|\,\prod_{ i\leq t} (\alpha+\beta\mu_i).
$$
But $|\bfv|=n^{-1/2}$ and $|\bfu|=n^{1/2}$, so we have established (\ref{bind}).
\end{proof}

\Section{Invertibility and One-Way Functions}
\label{oneway}

In this section we present definitions and notation concerning one-way functions.

For a set $S$, $\ident_S$ is the identity function on $S$.
Consider functions $F{:}\ S\rightarrow T$ and $G{:}\ T\rightarrow S$.
We say $G$ is a {\em right inverse} of $F$ if $F\compose G=\ident_T$ and
is a {\em left inverse} of $F$ if $G\compose F=\ident_S$.\footnote{Other
common terms for left inverse are {\em retract} and {\em retraction}.
Other common terms for right inverse are {\em coretraction } and {\em section}.}

 We say $F$ is a  {\em partial function} from  $S$ to $T$
(written $F{:}\ S\parfun T$) if it maps a subset of $S$
(the {\em domain} of $F$, denoted $\dom(F)$) onto a subset of
of $T$ (the {\em range} of $F$, denoted $\ran(F)$).
If $F{:}\ S\parfun T$ and $G{:}\ T\parfun U$, the {\em composition}
of $F$ and $G$, written $G\compose F$,  is a partial function
from $S$ to $U$ that maps $s$ to $u$ precisely when there is a $t$ such
that $F(s)=t$ and $G(t)=u$.
For  partial functions $F{:}\ S\parfun T$ and $G{:}\ T\parfun S$,
$G$ is a {\em partial right inverse} of $F$ if $F\compose G=\ident_{\ran(F)}$ and
is a {\em partial left inverse} of $F$ if $G\compose F=\ident_{\dom(F)}$.

Like functions, partial functions are injective if and only if
they have a (partial) left inverse.  Unlike functions,
which are surjective if and only if they have a
left inverse, 
partial functions always have a partial right inverse.\footnote{The existence of partial right inverses for
partial functions is equivalent to the Axiom of Choice.  This is almost the
same result as Axiom of Choice equivalent AC~5 in Rubin and Rubin \cite{rubin;rubin:1970},
which says that  the existence of right inverses for onto functions is equivalent to the Axiom of Choice.  }
 In cryptography,  invertibility of a  partial function refers to existence of an {\em efficiently computable right
inverse} of some kind.

Let $\varphi(n)$ be a
proposition concerning the natural numbers  $n\in\nat=\{0,1,2,\ldots\}$.
We say that $\varphi(n)$ holds {\em infinitely often}, and write $\varphi(n)\ $i.o.
or $(\exists^\infty n)\,\varphi(n)$, if
$  \forall m\, (\exists n\geq m)\, \varphi(n)$.
We say that $\varphi(n)$ holds {\em almost always}, and write $\varphi(n)\ $a.a.
or $(\forall^\infty n)\,\varphi(n)$, if
$
   \exists m\, (\forall n\geq m)\, \varphi(n).
$
$\exists^\infty$ and $\forall^\infty$ are dual quantifiers:
$\lnot(\exists^\infty n)\, \varphi(n)$ is equivalent
to $(\forall^\infty n)\,\lnot\varphi(n)$ and $\lnot(\forall^\infty n)\, \varphi(n)$ is equivalent
to $(\exists^\infty n)\,\lnot\varphi(n)$.  It is helpful to keep this in mind
when negating statements.  

A function $p{:}\ \nat\rightarrow\real$ is {\em negligible} if $p(n)=n^{-\omega(1)}$,
\ie, $(\forall c>0)\,(\forall^\infty n)\,(|p(n)|\leq n^{-c})$. 
(Some sources use the term {\em superpolynomially small} rather than {\em negligible}.)
In contrast, $p$ is {\em significant}
if $p(n)=n^{-O(1)}$, \ie, $(\exists c>0)\,(\forall^\infty n)\,(p(n)\geq n^{-c})$.
 The term {\em negligible} is standard in cryptographic theory;  the term {\em significant} is not.

Write $p(n)\approx q(n)$ if $|p(n)-q(n)|$ is negligible and $p(n)\gg q(n)$ if $p(n)-q(n)$ is significant.

Let $F$ be a function from $\{0,1\}^*$ to $\{0,1\}^*$.  Define the {\em auxiliary function} of $F$ to be $\olf(x)=( 1^{|x|},F(x)),$
where $1^{|x|}$ is the unary representation of the input length.
Auxiliary functions are convenient when defining weakly and strongly one-way
functions because they inform an adversary attempting to invert $F$ what the 
length of the preimage is.  It is reasonable to suppose that an adversary would
have this information.

We will take a probabilistic
approach where the arguments of $F$ are uniformly distributed random bit strings
$X_n\in\{0,1\}^n$ and  $Y_n=F(X_n)$.  $Y_n$
may not be  uniformly distributed;  indeed, it need not have a fixed length for a given $n$.

We require another modification: a probabilistic adversary.
Thus, the adversary attempting to find $X_n$ such that $\olf(X_n)=(1^n,Y_n)$ is a  partial function $\olg(1^n,Y_n,R_n)$
computable in time polynomial in $|(1^n,Y_n)|$, where  $R_n$ is a
random bit string independent of  $X_n$.  We may assume that $R_n$ is uniformly
distributed on $\{0,1\}^{q(n)}$ for some polynomial $q(n)$.  We will say that $\olg$ is
a {p.p.t.} function in this circumstance.

For each $n>0$ and polynomial $q(n)$, $(1^n,Y_n,R_n)$ is a random  vector and
$\prob_{(1^n,Y_n,R_n)}$, denoted for the
sake of simplicity as $\prob^n$, is an induced probability measure on $\Omega' =\{1\}^*\times\{0,1\}^*\times\{0,1\}^*$.  (Strictly speaking,
this  notation should specify the particular polynomial
$q(n)$ used, but this is cumbersome.)
In the definitions below, the function $I(1^n,y,r)=(1^n,y)$ acts as an identity function
when the adversary is the {p.p.t.} function $\olg$.

\begin{definition}
A polynomial time computable
function $F:\{0,1\}^*\rightarrow\{0,1\}^*$ is a {\em weakly one-way function} if 
$$
  (\exists \delta\gg 0)(\forall\,\mathrm{{p.p.t.}}\, \olg)\  (\prob^n\{\olf\compose \olg=I\}\leq 1-\delta\ \mathrm{a.a.}),
$$
or equivalently,
there is a $c>0$ such that for all p.p.t.\ $\olg$, $\prob^n\{\olf\compose \olg=I\}\leq 1-n^{-c}\ \mathrm{a.a.}$

A polynomial time computable
function $F$ is a {\em strongly one-way function} if
$$
  (\forall\,\mathrm{{p.p.t.}}\, \olg)\  (\prob^n\{\olf\compose \olg=I\}\approx 0),
$$
or equivalently,
for all $c>0$ and all p.p.t.\ $\olg$, $\prob^n\{\olf\compose \olg=I\}\leq n^{-c}\ \mathrm{a.a.}$
If, in addition to either of the conditions above, $F$ is a length-preserving permutation (\ie,
it is a bijection on $\{0,1\}^n$ when restricted to strings of length $n$), we say it is a {\em weakly}
or {\em strongly one-way permutation}.
\end{definition}

\begin{remark}
The notation used in this paper differs from the notation used
in other sources.  A typical definition (similar to the one
found in \cite{goldreich:2001}) says that $F$ is weakly 
one-way if there is a $c>0$ such that for
every {p.p.t.} $G$ and all large $n$,
$$
   \prob\{G(1^n,F(x))\notin F^{-1}F(x)\}>n^{-c}
$$
where the probability is taken uniformly over $x\in\{0,1\}^n$ and
the random bits used by $G$.  This is equivalent to the definition 
above.
\end{remark}

Given $\olf$ and  $\olg$ as above, let
$$W=\prob^n\{\olf\compose \olg=I|I\}.$$
$W(1^n,y)$ is the probability that $\olg$ successfully
finds a length  $n$ inverse of $y$.  Taking
the conditional expectation with respect to $I$ averages over the random bit strings used by $\olg$.  Using this notation, we
may prove a
 standard amplification result which yields a useful tail bound for $W$ when $F$ is a weakly one-way
function \cite{goldreich:1990,goldreich:2001}.

\begin{proposition}
\label{wowprop}
Let $\olf$ be an auxiliary  function and $\olg$ be a randomized partial function with
$W=\prob^n\{\olf\compose \olg=I|I\}$.
\begin{itemize}
\renewcommand{\theenumi}{(\roman{enumi})}
\upitem{i}   Let $\olg'$ be the randomized
partial function computed
by independently computing  $\olg$ $k$ times  on a given input $(1^n,y)$
(with fresh random bits each time) and returning the first value
$x$ such that $\olf(x)=(1^n,y)$ (if there is such a 
value).  Let $W'=\prob^n\{\olf\compose \olg'=I|I\}$. Then
for any $0<\varepsilon<1$,
$$
  \prob_I^n\{W>\varepsilon \} = \prob_I^n\{W'>1-(1-\varepsilon)^k\}.
$$
\upitem{ii}
If  $F$ is a weakly one-way function, where for some significant $\delta$, every {p.p.t.} $\olg$ satisfies   $\prob^n\{\olf\compose \olg=I\}\leq 1-\delta$, then for every significant  $\varepsilon=\varepsilon(n)$,
$$
\prob^n_I\{W>\varepsilon \}\leq 1-\delta/2\ \mathrm{a.a.}
$$
\end{itemize}
\end{proposition}
\begin{proof} 
Fix a value $(1^n,y)$ in the range of $\olf$.  
$W(1^n,y)>\varepsilon $ asserts that  the probability that  $\olf\compose\olg(1^n,y)=(1^n,y)$
is greater than $\varepsilon$ or,  equivalently, the probability that
$\olf\compose\olg(1^n,y)\neq (1^n,y)$ is  less than $1-\varepsilon$.
This happens  if and only if the probability that  $k$ independent computations of $\olf\compose\olg(1^n,y)$ fail to yield $(1^n,y)$  is less than $(1-\varepsilon)^k$;
and this happens if and only if  $W'(1^n,y)> 1-(1-\varepsilon)^k$.
  This proves part~(i).

Now suppose $\varepsilon$ is significant. There is an integer $d$ 
such that $\varepsilon>n^{-d}\ \mathrm{a.a.}$
 In part (i) take $k=n^{d+1}$ so almost always
\begin{eqnarray*}
\prob^n_I\{W>\varepsilon \}&=& \prob_I^n\{W'>1-(1-\varepsilon)^k\}\\
							&\leq& \prob_I^n\{W'>1-(1-n^{-d})^{n^{d+1}}\}\\
							&\leq& \prob_I^n\{W'>1-e^{-n}\}.
\end{eqnarray*}
By Markov's inequality and the weakly one-way assumption,
we have almost always
\begin{eqnarray*}
  \prob^n_I\{W' > 1-e^{-n}\} &\leq& \expec(W')/(1-e^{-n})\\
               &\leq& (1-\delta)/(1-e^{-n})\\
               &=&     1-\delta+e^{-n}(1-\delta)/(1-e^{-n})\\
               &\leq&  1-\delta+e^{-n}\\
               &\leq& 1-\delta/2,
\end{eqnarray*}
from which  (ii) follows.
\end{proof}

\begin{remark}
Rephrased, Proposition~\ref{wowprop}(ii) says that 
$\prob_I^n\{W\leq\varepsilon \}\geq \delta/2\ \mathrm{a.a.}$, i.e.,  under
the hypotheses of the proposition, $W$   has a uniformly significant tail probability.
\end{remark}

We often deal with partial functions defined only on arguments of certain prescribed lengths.
The following  technical result shows that under certain circumstances we can obtain weakly and strongly one-way functions
(and permutations) from hard-to-invert partial functions by filling in undefined values.

\begin{proposition}
\label{partialow}
Let $F{:}\{0,1\}^*\parfun\{0,1\}^*$  be polynomial time computable with domain $\bigcup_{m\geq 0} \{0,1\}^{\tau(m)}$, where
$\tau(m)$ is strictly increasing, computable in time polynomial in $m$, and for some $k>0$, $\tau(m+1)\leq\tau(m)^k$.
Define $F'(x)$ on  strings $x$ of length $n$ as follows. Let $m$  be the
largest integer such that $\tau(m)\leq n$,   put $x=x'z$ where $|x'|=\tau(m)$,
and set $F'(x)=F(x')z$ (the concatenation of $F(x')$ and $z$).
\begin{itemize}
\upitem{i}  
%
If there is a $c>0$ such that for all p.p.t.\ $\olg$, $\prob^{\tau(m)}\{\olf\compose \olg=I\}\leq 1-\tau(m)^{-c}$ a.a., then $F'$ is a weakly one-way function.

\upitem{ii}  
If for all $c>0$ and all p.p.t. $\olg$, $\prob^{\tau(m)}\{\olf\compose \olg=I\}\leq \tau(m)^{-c}$ a.a., then $F'$ is strongly one-way.

\upitem{iii}  If for every $m\geq  0$, $F$  is a permutation (\ie, its restriction to each domain $\{0,1\}^{\tau(m)}$ is a bijection), then
$F'$ is also a permutation.
\end{itemize}
\end{proposition}

\begin{proof}  Part (iii) is obvious.

For parts (i) and (ii)
we derive an upper bound for $\prob^n\{\olf'\compose \olg'=I\}$, where  $\olg'$ is an arbitrary {p.p.t.} function.

From $\olg'$  construct a {p.p.t.} function $\olg$ that attempts to invert $\olf$ as follows on input $(1^{\tau(m)},y)$. It takes
successive values of  $n$ 
in the interval 
$\tau(m)\leq n< \tau(m+1)$, each time choosing a random
bit string $z$ of length $n-\tau(m)$ and applying
$\olg'$ to  $(1^n, yz)$; if the result is of the form
$xz$ and $F(x)=y$, it  returns the value $x$ and terminates the computation.

In (i) of the proposition, there is a $c>0$ such that
$$
  \prob^{\tau(m)}\{\olf\compose \olg=I\}\leq 1-\tau(m)^{-c}\ \mathrm{a.a.}
$$
$\olg$ finds an inverse image $x$ (with respect to $\olf$) of $(1^{\tau(m)},y)$
only if there is an $n$ in the range $\tau(m)\leq n<\tau(m+1)$ such that
$\olg'$ finds an inverse image $xz$ of $(1^n,yz)$.
But $1-\tau(m)^{-c}\leq 1-n^{-c}$, so
$$
  \prob^{n}\{\olf'\compose \olg'=I\}\leq 1-n^{-c}\ \mathrm{a.a.}
$$
Therefore, $F'$ is a weakly one-way function.

In part (ii) of the proposition, for all $c>0$ and all {p.p.t.} $\olg$,
$$
  \prob^{\tau(m)}\{\olf\compose \olg=I\}\leq \tau(m)^{-c}\ \mathrm{a.a.}
$$
Again, $\olg$ finds an inverse $x$ of $(1^{\tau(m)},y)$
only if there is an $n$ in the range $\tau(m)\leq n<\tau(m+1)$ such that
$\olg'$ finds an inverse $xz$ of $(1^n,yz)$.
But $n<\tau(m+1)\leq \tau(m)^k\ \mathrm{a.a.}$, so
$$
  \prob^{n}\{\olf'\compose \olg'=I\}\leq n^{-c/k}\ \mathrm{a.a.}
$$
and, therefore, $F'$ is a strongly one-way function.
\end{proof}

\Section{From Weakly to Strongly One-Way.}
\label{application}

We now show  that the existence of a weakly one-way function implies
the existence of a strongly one-way function.  As noted in the introduction,
the first published proof \cite{goldreich:2001} implicitly defines
a reduction $({\mathcal R},{\mathcal R^*})$ between function inversion problems.  
Let us make this precise.

Recall that from $F$ we define $F'$ by
$$
  F'(x_0x_1\cdots x_{t-1})=(F(x_0),F(x_1),\ldots,F(x_{t-1}))
$$
where $t=t(n)$ is a polynomially bounded function and $|x_i|=n$ for all $i$.
This defines $F'$ only for inputs $x'=x_0x_1\cdots x_{t-1}$ of length $n\,t(n)$,
but the Proposition~\ref{partialow} allows us to extend $F'$ to a total function.  It is convenient to view 
${\mathcal R}$ as a polynomial time oracle Turing machine  that computes
the  auxiliary function $\olf'$ by making queries to evaluate $\olf$
at the values $x_0,x_1,\ldots ,x_{t-1}$.  Similarly, 
${\mathcal R^*}$ is a probabilistic polynomial time oracle Turing machine  that computes 
$\olg$ by making queries to evaluate the function $\olf$ and the randomized
function $\olg'$ at various values.

We now show how the 
 probability inequality in Theorem~\ref{mainthm}(i)
 figures in Goldreich's 
proof   \cite{goldreich:2001}.

\begin{theorem}
\label{weaktostrong}
Suppose $F$ is a weakly one-way function, so that for some integer $c>0$,
$$
  (\forall\,\mathrm{{p.p.t.}}\, \olg)\  (\prob^n\{\olf\compose \olg=I\}\leq 1-n^{-c}\ \mathrm{a.a.}).
$$
Then $F'$, defined as above with $t=n^{c+1}$, is a strongly one-way function.
\end{theorem}

\begin{proof}
Take an arbitrary {p.p.t.} $\olg'$.
By Proposition~\ref{partialow}(ii) it is enough to show that for  all $d>0$,
$$\prob^{nt}\{\olf'\compose\olg'=I\}\leq (nt)^{-d} \ \mathrm{a.a.}$$
Let $Z=\prob^{nt}\{\olf'\compose\olg'=I|I\}$ so $\expec(Z)=\prob^{nt}\{\olf'\compose\olg'=I\}$ by the law of iterated expectations.

${\mathcal R^*}$ computes a randomized function $\olg$, which attempts to find the inverse of $\olf$ 
at $(1^n,y)$, as follows.
\begin{enumerate}
\item ${\mathcal R^*}$ forms  ${\mathbf y}=(y_0,y_1,\ldots,y_{t-1})$ by 
choosing a random $i<t$ and
setting $y_i=y$; then generating random $x'_j\in\{0,1\}^n$ for each $j\not=i$ and putting  $y_j=F(x'_j)$.  (Here it queries the  $\olf$ oracle).

\item It queries the $\olg'$ oracle on  $(1^{nt}, {\mathbf y})$ and
receives an answer $x_0x_1\cdots x_{t-1}$,
where each $x_j$ is of length $n$.

\item It checks that   $F(x_j)=y_j$  for all $j<t$ (again, by querying the
$\olf$ oracle) and,
if so, returns $x_i$;  otherwise, the function is undefined.
\end{enumerate}
Clearly, when this procedure returns a value $x_i$, $(1^n,y)=\olf(x_i)$.

Define $U_i(1^{nt},{\mathbf y})=(1^n,y_i)$.  The random objects $U_i$ are {i.i.d.}  The probability
that $\olg'$ successfully finds an inverse for $(1^{nt},{\mathbf y})$ given that $y=y_i$ is $W_i(y)$, where
$W_i=\expec(Z|U_i)$.  Therefore, $\prob^n\{\olf\compose\olg=I|I\}$ is 
$W=(\sum_{i<t} W_i)/t$.  It is at this point we bypass many of the details in the original proof
and invoke Theorem~\ref{mainthm}(i):  $$\expec(Z)\leq\prob_I^n\{W>\varepsilon \}^t+t\varepsilon.$$
Let $d$ be an arbitrary positive integer and put  $\varepsilon=n^{-d}t^{-d-1}/2$.
By Proposition~\ref{wowprop}, $\prob_I^n\{W>\varepsilon \}\leq 1-n^{-c}/2\ \mathrm{a.a.}$
Thus, almost always
\begin{eqnarray*}
\prob_I^n\{W>\varepsilon \}^t&\leq& (1-n^{-c}/2)^{n^{c+1}}\\
                    &\leq& (e^{-n^{-c}/2})^{n^{c+1}}\\
                    &=&    e^{-n/2}\\
                    &\leq& (nt)^{-d}/2.
\end{eqnarray*}
Also, $\varepsilon t=(nt)^{-d}/2$.
Consequently, $\prob^{nt}\{\olf'\compose\olg'=I\}=\expec(Z)\leq (nt)^{-d}\ \mathrm{a.a.}$ 
\end{proof}

Using Proposition~\ref{partialow}, we have the following corollary.

\begin{corollary}~
\begin{itemize}
\upitem{i} If weakly one-way functions exist, then strongly one-way functions exist.

\upitem{ii} If weakly one-way permutations exist, then strongly one-way permutations exist.
\end{itemize}
\end{corollary}

\Section{A Security Preserving Reduction for One-Way Permutations.}
\label{security}

In this section we restrict our attention to one-way permutations.  Since
$|F(x)|=|x|$, it is not necessary to 
to use an auxiliary function $\olf(x)=(1^{|x|},F(x))$ -- an adversary can infer this
information from $|F(x)|$.

Definitions of  {\em security preserving reduction} differ on details
\cite{goldreich:1990,crescenzo:99,lin:2005,goldreich:2011,goldreich:2011a},
but follow the same general pattern.  $(\mathcal{R,R^*})$ is a reduction
from a cryptographic primitive $F$ to a cryptographic primitive  $G$, as in the
previous section, but in addition, when $\mathcal{R^*}$ computes $F'$ from
$G'$, 
the security of $F'$ against $G'$ is at least as great as
 the security of $F$ against $G$.
For  one-way permutations, security means resilience against inversion.
To be precise, if $G$ is a  randomized function  computed in time $T(n)$  (not necessarily a polynomial)
and $\varepsilon(n)$ is  the probability
that $G$  inverts $F$, the  {\em security} of $F$ against  $G$ is $S(n)=T(n)/\varepsilon(n)$.  
This is essentially the expected time to  invert some element with respect  to  $F$  
by  applying  $G$ independently to random elements
in the range of $F$.    
We will write
$S'(n)\succeq S(n)$ ($S'(n)$ {\em dominates} $S(n)$) if there are positive constants $k$ and $c$ such
that $S'(cn)\geq S(n)/n^k$ a.a., and   $S'(n)\asymp S(n)$ ($S'(n)$ and $S(n)$
are {\em of the same order})
 if $S'(n)\succeq S(n)$ and
$S(n)\succeq S'(n)$.

The reduction used in Theorem~\ref{weaktostrong} is not
security preserving because
$\mathcal R^*$ makes just one query to the
$G'$ oracle and it is of length $n^{c+2}$.  If the security of $F'$ against $G'$ is $S'(n)$,  then the security of $F$ against $G$ is 
of the same order as  $S'(n^{c+2})$, which is not of the same order as $S'(n)$ in general.

We apply Theorem~\ref{mainthm}(ii) to show that the reduction of
Goldreich \etal\ \cite{goldreich:1990} is security preserving, in fact, that
 $S(n)$ is essentially  $S'(n+\omega(\log n))$, which is of the same
 order as $S'(n)$.
This reduction applies just to one-way permutations rather than arbitrary one-way functions.
It uses the
 set of $t$-walks in a hybrid expander-permutation directed graph
in place of  a direct power  and $\beta$-independence  in place of independence.

The expander graphs 
used for this reduction must be from a fully explicit family, defined as  follows.  
\begin{definition}
Let  ${\mathcal G}_m, m\geq 0,$ be 
a family of $d$-regular graphs where ${\mathcal G}_m$ has
vertex set $V_m=\{0,1,\ldots,N_m-1\}$  with
 $N_0<N_1<N_2<\cdots$. 
A {\em rotation function}  $R(N,u,j)$ for this family is a partial function satisfying the following
conditions.  
\begin{enumerate}
\item $R(N,u,j)$ is defined if and only if for some $m$, $N=N_m$,
$0\leq u<N_m$, and $0\leq j<d$.

\item  
For each vertex  there is a linear order on the $d$ adjacent vertices such that $R(N_m,u,j)=(v,k)$ holds precisely when
$v$ is the $j$-th vertex adjacent to $u$ and $u$ is the $k$-th vertex adjacent to $v$.
 \end{enumerate}
A family of $d$-regular graphs is {\em fully explicit} if it
has a polynomial time computable rotation function.  

If  ${\mathcal G}_m, m\geq 0,$
has a rotation function $R(N,u,j)$ such that
for every edge $\{u,v\}$ in ${\mathcal G}_m$ there
is  a $j$ such that $R(N_m,u,j)=(v,j)$, we let $\kappa(u,v)=j$.
Thus, $\kappa$ defines an {\em edge coloring}: $k$ assigns a color to 
each edge such that at each vertex incident edges are uniquely colored.
When this occurs for a polynomial time computable $R(N,u,j)$ we
have a {\em fully explicit edge coloring}.
\end{definition}  

There is an extensive  literature on the construction of fully explicit families of $(N,d,\alpha)$-expander graphs 
\cite{margulis:1973,gabber:1981,jimbo:1987,reingold;vadhan;wigderson:2002,alon;schwartz;shapira:2008,benaroya;tashma:2008}.

\begin{remark}
Most applications involving fully explicit expander graph 
families require that the gap between 
$N_m$ and $N_{m+1}$ not grow too quickly.
We require more, \viz, that
$N_0,N_1,N_2,\ldots $  be a smoothly growing sequence of powers of two 
with $N_m=2^{cm}$ for some constant $c$,   and that $d$  be a fixed 
power of two, say $2^e$.
For the remainder of this section we will assume that
${\mathcal G}_m, m\geq 0,$ is a fully explicit
$(N,d,\alpha)$-expander graph family satisfying these conditions, with $d$ fixed and $\alpha<1$. 
Hence, ${\mathcal G}_m=(\{0,1\}^n,E_m)$
 with $n=cm$.  As a notational convenience, we  will take $E=\bigcup_{m\geq 0}E_m$ and
write ${\mathcal G}_m=(\{0,1\}^n,E)$ rather ${\mathcal G}_m=(\{0,1\}^n,E_m)$.
\end{remark}

One example of an explicit family of expander graphs satisfying these conditions is the affine torus expander graph family  of Margulis \cite{margulis:1973}.
Gabber and Galil \cite{gabber:1981} established an upper
bound for the second largest eigenvalue magnitude of graphs
in this family, later improved by  Jimbo and Maruoka \cite{jimbo:1987}.  Using these results, we may take $n=2m$ (so $N_m=2^{2m}$),
$d=8=2^3$, and $\alpha=5\sqrt{2}/8= 0.88388\cdots$.  Other constructions may give a better 
bound for $\alpha$.  The results in \cite{goldreich:1990} require
that $\alpha\leq 1/2$, but the approach here
based on $\beta$-independence works for any fixed bound less than~1. 
 
Goldreich \etal\ \cite{goldreich:1990} require an expander graph 
family which has a fully explicit edge coloring  (but use different terminology).
Many explicit expander graph constructions do, in fact,
 have a fully explicit edge coloring, but we will extend the proof in \cite{goldreich:1990}  so that we may 
 dispense with this assumption.

Let $t=t(n)$ be a polynomially bounded, strictly increasing function.  
We first describe the transformation ${\mathcal R}$ taking
$F$, a weakly one-way permutation, to $F'$, a slightly harder to invert permutation.

Take  $E'=F\compose E$.  This gives a family  of directed graphs ${\mathcal G}'_m=(\{0,1\}^n,E')$.
Note that for each directed  edge $(u,w)$ in ${\mathcal G}'_m$, 
there is a unique $v$ such that $\{u,v\}\in E$ and $F(v)=w$.
This suggests two ways to color the directed edges of ${\mathcal G}_m'$.
If $R(N_m,u,j)=(v,k)$, we have the coloring
$\kappa(u,w)=j$ and the coloring $\kappa'(u,w)=k$.
Thus, $\kappa$  is an explicit out-edge coloring in the sense that at every vertex $u$ of ${\mathcal G}_m'$, the $d$  out-edges are uniquely colored; and $\kappa'$  is an explicit in-edge coloring in the sense that at every vertex $w$, the
$d$  in-edges are uniquely colored.  

Let ${\mathbf x}=(x_0,x_1,\ldots,x_t)$ be a directed $t$-walk in ${\mathcal G}_m'$. 
The {\em forward representation}  of $\mathbf x$ is 
\begin{eqnarray*}
\varphi({\mathbf x})&=& (x_0,\kappa(x_0,x_1),\kappa(x_1,x_2),\ldots,\kappa(x_{t-1},x_t)).
\end{eqnarray*}
In effect, we regard ${\mathcal G}_m'$ and its
out-edge-coloring as a finite automaton with alphabet
$\{0,1,\ldots,d-1\}$, and  $\kappa(x_0,x_1)\kappa(x_1,x_2)\cdots\kappa(x_{t-1},x_t)$  as
the unique string  which, when read, causes this finite automaton to   transition
through the states $x_0,x_1,\ldots,x_t$.  This walk representation uses 
fewer bits than just listing  vertices.
Clearly, $\varphi$ is a bijection from the set of directed $t$-walks in $\mathcal G'$
to $V\times \{0,1,\ldots,d-1\}^t$.   We will identify $V\times \{0,1,\ldots,d-1\}^t$
with $\{0,1\}^{n+te}$. 
Since $R$ and $F$ are polynomial time computable, so are $\varphi$ and $\varphi^{-1}$.  

The {\em reverse representation} of ${\mathbf x}$ is
\begin{eqnarray*}
\rho({\mathbf x})&=&  (x_t,\kappa'(x_{t-1},x_t),\kappa'(x_{t-2},x_{t-1}),\ldots,\kappa'(x_0,x_1)).
\end{eqnarray*}
View this as taking
${\mathcal G}_m'$ together with its in-edge-coloring,  reversing the  edge directions to 
form another  finite automaton, and specifying a succinct walk representation as before.
As with $\varphi$, $\rho$ is a bijection from the set of directed $t$-walks in $\mathcal G'$
to $\{0,1\}^{n+te}$.  It is easy to see that $\rho$  is polynomial time computable, but since
$F$ is weakly one-way,  it does not follow that $\rho^{-1}$ is polynomial time computable.

 We now describe how
 ${\mathcal R}$ computes $F'$, a permutation on $\{0,1\}^{n+te}$, from $F$, a permutation
 on $\{0,1\}^n$, where $n=cm$.
For each $t$-walk $\bf x$ in ${\mathcal G}'_m$, $F'$
maps $\varphi({\mathbf x})$ to $\rho({\mathbf x})$.  
In other words, $F'=\rho\compose\varphi^{-1}$.
To compute $F'(x_0,k_1,\ldots,k_t)$, $\mathcal R$ begins
at vertex $x_0$ in ${\mathcal G}_m$, then repeatedly
follows the $k_i$-th edge with respect to the coloring $\kappa$  from the current vertex
$x_i$ and applies $F$
to jump to a new vertex $x_{i+1}$.  In this way, it computes a $t$-path ${\mathbf x}=(x_0,x_1,\ldots,x_t)$.
As it traverses each edge  in ${\mathcal G}'_m$ dictated the out-edge coloring $k_i$ (with respect
to $\kappa$),  it computes the corresponding in-edge coloring {(with respect to $\kappa'$),
then at the terminus $x_t$ reverses the in-edge color sequence to output $\rho({\mathbf x})$.
Clearly,
$F'$   is a permutation.  By this 
definition, $F'$ is defined only on strings of length
 $cm+te$ for $m\geq 0$, but we may use
Proposition~\ref{partialow} to extend $F'$ so that it is defined
on strings of any length.

\begin{lemma}
\label{securelem}
From a weakly one-way permutation $F$, construct $F'$ as above with
polynomially bounded $t=t(n)$.  Suppose  $\delta=\delta(n)$ is significant and  that for
all p.p.t.\ $G$, $$\prob^n\{F\compose G=I\}\ \leq 1-\delta\ \mathrm{a.a.}$$  Then
the following hold.
\begin{itemize}
\upitem{i}  For every p.p.t.\ $G'$, 
$$
  \prob^{n+te}\{F'\compose G'=I\} \leq (1-\beta\delta(n) /2)^t\ \mathrm{a.a.}
$$
\upitem{ii} If $t\geq  7/\beta$, then for every p.p.t.\ $G'$,
$$
   \prob^{n+te}\{F'\compose G'=I\} \leq \max(1-2\delta(n) , 1/2)\ \mathrm{a.a.}
$$ 
\upitem{iii} If $\delta(n)\geq 1/2$ a.a. and $t=\omega(\log n)$, then for every p.p.t.\ $G'$,
$$
   \prob^{n}\{F'\compose G'=I\} \approx 0.
$$ 
\end{itemize}

\end{lemma}

\begin{proof}
(i) Let $Z=\prob^{n+te}\{F'\compose G'=I|I\}$ so $\expec(Z)=\prob^{n+te}\{F'\compose G'=I\}$.
As in the proof of Theorem~\ref{weaktostrong}, we have
$\mathcal R$  taking $F$ to  $F'$ and must specify 
 $\mathcal R^*$ taking each   p.p.t.\ function $G'$, which attempts to invert $F'$, to another p.p.t.\ function $G$, which attempts to invert $F$.  On a given input $y$,  $\mathcal R^*$, querying oracles
for $G'$ and $F$,  computes $G$ as follows.

\begin{enumerate}
\item   $\mathcal R^*$ chooses a random $i$ in the interval $1\leq i< t$,
then generates a random sequence of integers $k_1,k_2,\ldots,k_{t-i}$,
where each $k_j$ is in the range $0\leq k_j< d$;
$(y,k_1,k_2,\ldots,k_{t-i})$ is the forward walk representation of 
a random $(t-i)$-walk  in ${\mathcal G}_m'$ with initial vertex~$y$.  
$\mathcal R^*$  applies
$F$ to obtain the reverse walk representation
$(y_t, k'_1,k'_2,\ldots,k'_{t-i})$ then generates
random integers $k'_{t-i+1},,k'_{t-i+2},\ldots,k'_t$ in the range $0\leq k'_j<d$
to obtain ${\mathbf y}=(y_t,k'_1,k'_2,\ldots,k'_t)$, the reverse walk representation
of a  $t$-walk $(y_0,y_1,\ldots,y_t)$ chosen randomly from $t$-walks such that $y_i=y$.

\item  $\mathcal R^*$ queries the $G'$ oracle on  $ {\mathbf y}$ and
receives an answer $${\mathbf x}=(x_0,k_1'',k_2'',\ldots,k_t'').$$

\item $\mathcal R^*$ applies $\varphi^{-1}$ to $\mathbf x$ to obtain
a purported walk $(x_0,x_1,\ldots,x_t)$, and then applies $\rho$.
If the result matches $\bf y$, $\mathcal R^*$ has verified that 
 $(x_0,x_1,\ldots,x_t)$ is indeed a $t$-walk and 
that $x_i=y$ and  $(x_{i-1},x_i)$ is an edge in ${\mathcal G}'_m$.
In this case,  there
is a $v$ such that $\{x_{i-1},v\}$ is an edge in ${\mathcal G}_m$ and
$F(v)=x_i$, so $\mathcal R^*$ returns $v=R(2^n,x_{i-1},k_i'')$;
otherwise, the function is undefined.
\end{enumerate}
When this procedure returns a value $v$, $F(v)=y$.  
\bigskip

Define $U_i({\mathbf y})=y_i$, where $\rho^{-1}({\bf y})=(y_0,y_1,\ldots,y_t)$.  
By Theorem~\ref{beta} and the bijectivity of $\rho$, the random objects $U_i$ are $\beta$-{i.i.d.}  The probability
that $G'$ successfully finds an inverse of ${\mathbf y}$, given that $y=y_i$, is $W_i(y)$, where
$W_i=\expec(Z|U_i)$.  Therefore, $\prob^n\{F\compose G=I|I\}$ is 
$W=(\sum W_i)/t$.
By Theorem~\ref{mainthm}(ii),  $$\expec(Z)\leq(\alpha+\beta\,\prob_I^n\{W>\varepsilon \})^t+t\varepsilon.$$
By Proposition~\ref{wowprop}, $\prob_I^n\{W>\varepsilon \}\leq 1-\delta/2$ a.a.
Thus, almost always
\begin{eqnarray*}
(\alpha+\beta\,\prob_I^n\{W>\varepsilon \})^t&\leq& (\alpha+\beta(1-\delta/2))^t\\
                    &=& (1-\beta\delta/2)^t.
\end{eqnarray*}
Hence, for every significant $\varepsilon$, $\expec(Z)\leq (1-\beta\delta/2)^t+t\varepsilon$,
which proves (i).

\noindent (ii) The function $f(x)=(1-\beta x)^t$ is convex for $x\geq 0$ and
\begin{eqnarray*}
  f((\beta t)^{-1}) &=& (1-1 / t)^t\\  
                                 &<& e^{-1}
\end{eqnarray*}
so it lies below the continuous piecewise linear function 
$$
g(x)=\left\{
\begin{array}{ll}
 1-(1-e^{-1})\beta tx, &  \mbox{if $0\leq x\leq (\beta t)^{-1}$},\\
e^{-1}, &  \mbox{if $x> (\beta t)^{-1}$}.
 \end{array}
\right.
$$
In other words, $f(x)\leq g(x)=\max\left(1-(1-e^{-1})\beta t x,e^{-1}\right)$  (see Figure~1).
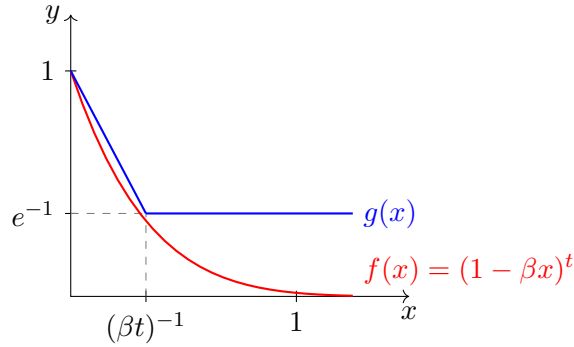
\begin{figure}[b]
\begin{center}
\begin{tikzpicture}[domain=0:1.25,xscale=3,yscale=3]
\draw[<->] (0,1.25) node[left]{$y$}-- (0,0) -- (1.5,0) node[below] {$x$};
\draw[-] (-.025,1)node[left]{$1$}--(.025,1);
\draw[-] (-.025,0.36788)node[left]{$e^{-1}$}--(.025,0.36788);
\draw[-] (1,-.025)node[below]{$1$}--(1,.025);
\draw[-] (0.33333,-.025)node[below]{$(\beta t)^{-1}$}--(0.33333,.025);
\draw[-,gray,dashed] (0,0.36788)--(0.33333,0.36788);
\draw[-,gray,dashed](0.33333,0)--(0.33333,0.36788);
\draw[red,thick] plot (\x, {(1-0.5*\x)^6}) node[above right] {$f(x)
=(1-\beta x)^t
$};
\draw[blue,thick,domain=0:0.33333] plot (\x, {1.0-1.89636*\x}) ;
\draw[blue,thick,domain=0.33333:1.25]
     plot (\x, {00.36788}) node[right] {$g(x)
$};
\end{tikzpicture}
\end{center}
\caption{Function $f(x)$ lies below $g(x$).}
\end{figure}
Thus, setting $x=\delta/2$ gives
$$
  (1-\beta\delta/2)^t \leq \max\left(1-(1-e^{-1})\beta t \delta/2,e^{-1}\right).
$$
But  $\beta t\geq 7$ and $(1-e^{-1})/2=0.31606\cdots$, so from part
(i) of the theorem
$$
   \prob^{n+te}\{F'\compose G'=I\} \leq \max(1-2\delta, 1/2)\ \mbox{a.a.}
$$

\noindent (iii) This also follows from part (i).  Since $\delta\geq 1/2$\ {a.a.},
 $1-\beta\delta/2\leq 1-\beta/4$\ a.a.  We know $t=\omega(\log n)$\
 a.a., so $(1-\beta\delta/2)^t$ is negligible. 
\end{proof}

To conclude, we show that the  reduction of Goldreich \etal\ \cite{goldreich:1990}
from a weakly one-way permutation to a strongly  one-way permutation  is security preserving.  

First, consider the reduction $(\mathcal{R}_0,\mathcal{R}_0^*)$  described in the proof of Lemma~\ref{securelem}(ii)
where $t\geq 1/\beta$ is an even integer.  (Recall that   $n=2m$ and 
$e=3$, so an even $t$  ensures that $n+te$ is  even.)  If $F$ is a weakly one-way
permutation, where there is a significant $\delta$ such that 
for all p.p.t.\ $G$ $\prob^n\{F\compose G=I\}\ \leq 1-\delta\ \mathrm{a.a.}$, then
applying $\mathcal{R}_0$ $s$ times to $F$ results in a
permutation $F'$ such that
for any p.p.t.\ $G'$ attempting to invert $F'$
$$
   \prob^{n+rte}\{F'\compose G'=I\} \leq \max(1-2^s\delta, 1/2)\ \mathrm{a.a}
$$ 
We know that for some $c>0$, $\delta(n)<n^{-c}$ a.a., so, setting $s=\lceil c\log n\rceil$, we have that
for every p.p.t.\ $G'$, $ \prob^{n}\{F'\compose G'=I\} \leq 1/2\ \mathrm{a.a}$.

Next, consider the reduction $(\mathcal{R}_1,\mathcal{R}_1^*)$  described in the proof of Lemma~\ref{securelem}(iii) where $t=\omega(n)$.  Applying
$\mathcal{R}_1$ to $F'$ results in a permutation $F''$ such that for
all p.p.t.\ $G''$ attempting to invert $F''$,
$$
\prob^n\{F''\compose G''=I\}\approx  0.
$$
Thus, a transformation  $\mathcal{R}$ consisting of $s$ applications 
of  $\mathcal{R}_0$ followed by an
application of $\mathcal{R}_1$ takes $F$ to $F''$, thereby
increasing input length from $n$ to $n+\omega(\log n)$.  
Transformation $\mathcal{R}^*$ consisting of an application
of $\mathcal{R}_1^*$ followed by $s$ applications of $\mathcal{R}_0^*$
takes p.p.t.\ function $G''$   to  p.p.t.\ function $G$,
thereby decreasing input length from $n+\omega(\log n)$ to $n$.
$\mathcal{R}^*$, computing $G$,  queries the oracle for $G''$ just once and 
the probability that $G$ inverts $F$ is precisely the probability that
$G''$ inverts $F''$.  $\mathcal{R}^*$ runs in polynomial time (assuming
constant time to answer a query).  Thus, we have the following result.

 \begin{theorem}
 \label{last}
 For reduction $(\mathcal{R},\mathcal{R}^*)$ described above,  if
 the security of $F''$ against $G''$ is $S''(n)$, then the security of
 $F$ against $G$ is $S(n)\asymp S''(n+\omega(\log n))$.  
 Thus, $S(n)\asymp S''(n)$.
 That is, $(\mathcal{R},\mathcal{R}^*)$ is a security preserving
 reduction taking weakly to strongly one-way functions.
 \end{theorem}
 
 \begin{remark}
 The  reduction  of Theorem~\ref{last}  takes every
 weakly one-way permutation to a strongly one-way permutation.
 In contrast, the proof of Theorem~\ref{weaktostrong} shows that for every weakly one-way function there is a reduction to a strongly one-way function. It 
  is not apparent from this proof that there is  just one reduction that takes
 every weakly one-way function to a strongly one-way function.  Thus,
 the reduction of Theorem~\ref{last} is   security preserving
 but is also  uniform in this sense.
\end{remark}

\Section{Conclusion.}
\label{conclusion}

 One-way functions are not the only cryptographic primitives proved secure using repetition
 or expander graph constructions
 to amplify hardness. Others include
collision-resistant hash functions \cite{canetti;sudan;trevisan;vadhan;wee:2007}, encryption schemes \cite{dwork;naor;reingold:2004}, weakly
verifiable puzzles \cite{canetti;halevi;steiner:2005,impagliazzo;jaiswal;kabanets:2009,jutla:2010}, signature schemes and
message authentication codes \cite{dodis;impagliazzo;jaiswal;kabanets:2009}, commitment schemes \cite{halevi;rabin:2008,chung;liu;lu;yang:2010},
pseudorandom functions and pseudorandom generators \cite{dodis;impagliazzo;jaiswal;kabanets:2009,maurer;tessaro:2010},
block ciphers \cite{luby;rackoff:1986,naor;reingold:1997,maurer;tessaro:2009,tessaro:2011}, and interactive protocols
\cite{pass;venkita:2012,hastad;pass;wikstram;pietrzak:2010,haitner:2013}.  
Another approach to constructing security preserving reductions  
 uses  hash functions rather than expander graphs to control  input size blowup 
 \cite{crescenzo:99,haitner;harnik;reingold:2006,bogdonov;rosen:2013}.  Conditional
 expectation bounds may help simplify proofs in some of these cases and point
 the way to  other cryptographic applications.

\end{document}